\documentclass{article}

\usepackage[ansinew]{inputenc} % ASCII (Western CP)
\usepackage{graphicx}
\usepackage{amsmath,amsthm,amssymb,amscd}
\usepackage{color}
\usepackage[colorlinks]{hyperref}
\usepackage[matrix,arrow,curve]{xy}

\title {A basis of the group of \\ primitive almost pythagorean triples}

\author{Nikolai A. Krylov\\ ~ \\
Siena college, department of mathematics\\
515 Loudon Road, Loudonville NY 12211\\ ~ \\
nkrylov@siena.edu}

\date {}

\begin{document}

\newtheorem{thm}{Theorem}
\newtheorem{lem}{Lemma}
\newtheorem{claim}{Claim}
\newtheorem{dfn}{Definition}
\newtheorem{prop}{Proposition}
\newtheorem{example}{Example}
\newtheorem{note}{Note}

\def\natu       {\mathbb N}
\def\inte       {\mathbb Z}
\def\rati       {\mathbb Q}
\def\real       {\mathbb R}
\def\comp       {\mathbb C}
\def\field      {\mathbb F}
\def\P          {{\cal P}}
\def\rP		{{\mathbb P}}
\def\F          {{\cal F}}
\def\A          {{\cal A}}
\def\C          {{\it Cl}}
\def\O          {{\cal O}}
\def\GCD        {{\rm gcd}}
\def\T		{{\cal T}}

\def\lla        {\longleftarrow}
\def\lra        {\longrightarrow}
\def\ra         {\rightarrow}
\def\hra        {\hookrightarrow}
\def\lmt        {\longmapsto}
\def\lam        {\lambda}
\def\del        {\delta}
\def\eps        {\epsilon}

\maketitle

\parskip=3mm

\begin{abstract}
Let $m$ be a fixed square-free positive integer, then equivalence classes of solutions of
Diophantine equation $x^2+m\cdot y^2=z^2$ form an infinitely
generated abelian group under the operation induced by the complex
multiplication. A basis of this group is constructed here using 
prime ideals and the ideal class group of the field $\mathbb Q (\sqrt{-m})$.
\end{abstract}


\section{Introduction and notations}

Fix an arbitrary  square-free integer $m > 1$. Then the set
of equivalence classes of solutions of Diophantine equation
\begin{equation}
\label{main} x^2 + m\cdot y^2 = z^2
\end{equation}
will have a group structure. Indeed, given two arbitrary solutions
$(a_1,b_1,c_1)$ and $(a_2,b_2,c_2)$ of (\ref{main}) we can produce
another one $(A,B,C)$ via the formulas
\begin{equation}
\label{operation} A:= a_1a_2 - mb_1b_2, ~~B:= a_1b_2 + a_2b_1, ~~
C:= c_1c_2.
\end{equation}
Using projectivization as an equivalence relation on such triples,
we obtain a well defined binary operation on the set of the
equivalence classes, and hence an abelian group, which will be
denoted by $\P_m$ and called the {\it group of primitive almost
pythagorean triples}. 

Here is an alternative description of this group using a bit more technical terms:
Consider the multiplicative subgroup, say ${\cal A}_m$,  of non-zero elements 
whose norm is a square of 
a rational number in the imaginary quadratic field $\rati(\sqrt{-m})$. 
The non-zero rational numbers $\rati^*$ will make a subgroup of ${\cal A}_m$ 
and it is easy to see that the corresponding factor group ${\cal A}_m/\rati^*$ is $\P_m$.   

This group has been considered by various
authors (see e.g. \cite{Baldisserri}, \cite{Eckert}, \cite{Kulzer},
\cite{Lemmermeyer}) since the operation (\ref{operation}) was
introduced by Taussky in \cite{Taussky} for ordinary pythagorean
triples. It is well known that the group is infinitely generated and
has torsion only when $m=3$. Description of a basis has
been given however only for particular values of $m$ (see
\cite{Baldisserri}, \cite{Eckert} and \cite{Kulzer}). The goal of this paper is to 
present a basis of the group $\P_m$ using prime ideals and the
ideal class group of $\rati(\sqrt{-m})$, for all square-free integers $m > 3$.

Let me first introduce some notations, which will be similar to
notations used in \cite{Kulzer}. Consider the set of all ordered
primitive triples $(a, b, c)\in \inte \times\inte \times \natu$ such
that $a^2 + m\cdot b^2 = c^2$. A triple $(a,b,c)$ is called
primitive when the $\GCD(a,b,c)=1$. Two such triples $(a, b, c)$ and
$(A, B, C)$ are said to be equivalent if $\exists m, n \in \inte
\setminus \{0\}$ s.t. $m(a, b, c) = n(A, B, C)$, where $m(a, b, c) =
(ma, mb, |mc|)$. This is clearly an equivalence relation. The
equivalence class of $(a,b,c)$ will be denoted by $[a,b,c]$ and the
set of all such classes is $\P_m$. Note that $[a,b,c] = [-a,-b,c]$,
but $[a,b,c] \neq [-a,b,c]$. Every equivalence class $[a,b,c] \in
\P_m$ can be represented uniquely by a primitive triple $(\alpha,
\beta, \gamma)$, where $\alpha > 0$ and hence one could refer to
primitive triples to describe elements of the group. Following the
operation (\ref{operation}) we see that for any two classes $[a, b,
c],~[A, B, C] \in \P_m$
$$
[a, b, c] + [A, B, C] = [aA-mbB, aB+bA, cC].
$$
For the proof of the following theorem see \cite{Baldisserri} or
\cite{Kulzer}.

\noindent {\bf Theorem.} $\P_m$ is an infinitely generated abelian
group for each square-free $m > 1$. The identity element is
$[1,0,1]$ and the inverse of $[a, b, c]$ is $[a, -b, c]=[-a,b,c]$.
$\P_m$ is torsion free when $m\neq 3$.

Note that Zanardo and Zannier in \cite{Zanardo} studied the group 
consisting of the set of equivalence classes of solutions of $x^2 + y^2 = z^2$ in the ring $\O_K$ of integers of a 
number field $K$ and described a basis for the torsion-free part of that group.
The group I consider here is a proper subgroup of the one discussed in \cite{Zanardo}.

From now on let's assume that we have a fixed, square-free integer $m>3$, and denote the imaginary 
quadratic field $\mathbb Q (\sqrt{-m})$ by $K$. The corresponding ring of integers will be denoted 
by $\O_K$ and the ideal class group of $K$ by $Cl(K)$. I will denote the set of 
all rational prime numbers by $\rP$ and the set of all prime ideals of $\O_K$ by $P(\O_K)$ respectively.

\section{Ideal classes of order 2}

\begin{lem}\label{order 2}
Let $c=2^{i_0}\cdot p_1^{i_1}\cdot\ldots\cdot p_k^{i_k}$ be the prime decomposition of a natural number $c$, where 
$2$ is inert or splits in $K$ if $i_0\neq 0$. Suppose that for each odd prime $p_i,~i\in\{1,2,\ldots, k\}$ there exists a 
prime ideal $Q_i$ of $\O_K$ s.t. $N(Q_i) = p_i$, and let $Q_0= \langle 2\rangle$ if 2 is inert in 
$K$ and $Q_0\cdot Q'_0 = \langle 2\rangle$ if 2 splits in $K$ (here $Q'_0$ is the conjugate of a prime ideal $Q_0$). 
If the class of $Q_0^{i_0} \cdot Q_1^{i_1}\cdot \ldots \cdot Q_k^{i_k}$ has order 2 in $Cl(K)$, 
then there exist integers $u$ and $v$ s.t. either $u^2+mv^2=c^2$ or 
$u^2 + mv^2 = (2c)^2$. 
\end{lem}
\begin{proof}
If $[Q_0^{i_0} \cdot Q_1^{i_1}\cdot \ldots \cdot Q_k^{i_k}]^2 = [1]$, then there exists $z\in\O_K$ s.t. 
$$
\prod\limits_{t=0}^k Q_t^{2i_t} = \langle z\rangle = \left\langle \frac{u+v\sqrt{-m}}{2^{1-\delta}}\right\rangle,
$$
with $u,~v\in\inte$; $\delta= 0$ if $-m\equiv 1\pmod{4}$, and $\delta=1$ if otherwise. If 2 splits or ramifies in $K$, 
we can take the norm of both ends of the above equality to obtain 
$4^{1-\delta}c^2 = u^2 + mv^2$. If 2 is inert in $K$, then we apply the previous step to the product of 
prime ideals $[Q_1^{i_1}\cdot \ldots \cdot Q_k^{i_k}]^2 = [1]$ to obtain 
$$
4\left(\frac{c}{2^{i_0}}\right)^2 = u^2 + mv^2,~~~\mbox{which implies}~~~(2c)^2 = (2^{i_0}u)^2 + m(2^{i_0}v)^2.
$$
\end{proof}

Since every prime ideal of $\O_K$ lies over a unique rational prime, and every prime $p\in\natu$ lies below at least one prime ideal 
of $\O_K$ (see Theorem 20 of \cite{Marcus}), we have the canonical projection $\mu:P(\O_K)\lra \rP$, which has a lifting map
$$
l:\rP \lra P(\O_K)~~\mbox{such that}~~\mu\circ l = Id_{\rP}.
$$ 
Hence if we pick such a lifting $l$ and compose it with another canonical projection $\pi:P(\O_K) \lra Cl(K)$ we obtain a map $f$ 
and the following commutative diagram.
$$
\xymatrix{
P(\O_K) \ar[r]_-{\mu}  \ar[d]_{\pi} & \rP\ar@/^/[dl]^-{f} \ar@/_1pc/@{-->}[l]_-l \\
Cl(K)&
}
$$

\begin{note} 
Map $f$ does depend on the choice of the lifting $l:\rP\to P(\O_K)$.
\end{note}

\begin{note}
Since $\pi$ is a multiplicative homomorphism from the group generated by $P(\O_K)$, we can extend 
such map $f$ in a natural way, to a multiplicative function from $\natu$ to $Cl(K)$.
\end{note}

\begin{note}\label{c=even}
If we have a primitive triple  $(a, b, c)$ s.t. $a^2 + mb^2 = c^2$ then an odd prime $p\mid c$ 
iff the Legendre symbol $\left(\frac{-m}{p}\right) = 1$ (see Lemma 1.7 of \cite{Cox}).  It is also trivial to see that if $c$ 
is even in the primitive triple above, then we must have $-m\equiv 1\pmod{4}$, and if $-m\equiv 5\pmod{8}$, 
then $4\nmid c$. 
\end{note}

\noindent Let me now introduce two special subsets of $Cl(K)$ and of $\rP$ respectively. 

\begin{dfn}
The set of all elements of the ideal class group of $K$ which have order at most 2, will be denoted by $E$, that is 
$$
E:=\{ [A]\in Cl(K)~ ~ |~ ~A~\mbox{is an ideal of}~\O_K ~\mbox{and}~[A]^2=[1]\}
$$
Clearly, $E$ is a subgroup of $Cl(K)$.
\end{dfn}

\begin{dfn}
A subset $L\subset \rP$ is defined as the subset of rational primes, which have the Kronecker symbol $1$ modulo $m$, that is
$$
L:=\{p\in\rP~|~\left(\frac{-m}{p}\right) = 1\},
$$
where $\left(\frac{-m}{p}\right)$ is the Legendre symbol if $p$ is odd, and $-m\equiv 1\pmod{8}$ if $p=2$. Moreover, we define
$$
L_0:=f^{-1}(E)\cap L,~~\mbox{where}~f^{-1}(E)~\mbox{denotes the preimage of}~E~\mbox{in}~\rP.
$$
\end{dfn}

\begin{claim}
If  $-m\equiv 1\pmod{8}$ and $m>16$, then $2\notin L_0$, and $E\subsetneq Cl(K)$.
\end{claim}
\begin{proof}
Since  $-m\equiv 1\pmod{8}$, we have the splitting of $\langle 2\rangle$ into the product of prime ideals
$$
\langle 2 \rangle = \left\langle 2,~\frac{1+\sqrt{-m}}{2}\right \rangle\cdot  \left\langle 2,~\frac{1-\sqrt{-m}}{2}\right\rangle = P\cdot P',
$$
and it's enough to show that $P^2 =\langle 4,~\frac{1-m+2\sqrt{-m}}{4}\rangle$ is not principal in $\O_K$. If one assumes that 
there are $x,y\in\inte$ s.t. $P^2 = \langle\frac{x+y\sqrt{-m}}{2}\rangle$, then we could find $a,b\in\inte$ s.t.
$$
4 = \frac{a+b\sqrt{-m}}{2}\cdot \frac{x+y\sqrt{-m}}{2},~~~\mbox{and hence}~~~16^2 = (a^2+mb^2)\cdot (x^2+my^2).
$$
It's easy to see that assumptions like $b=0$ or $y=0$ lead to a contradiction (for example, by using the norm), 
and if $by\neq 0$, then $ (a^2+mb^2)\cdot (x^2+my^2) >16^2$, which is also not allowed.
\end{proof}

\begin{note}
If $m=7$, then $P^2 =\langle\frac{-3+1\sqrt{-7}}{2}\rangle$, and $[-3,1,4]\in \P_7$. If $m=15$, 
then $P^2 =\langle\frac{1+1\sqrt{-15}}{2}\rangle$, and $[1,1,4]\in \P_{15}$. 
\end{note}

Recall that Theorem 2 of \cite{Kulzer} describes a basis of $\P_m$ for four special cases of $m$, which correspond to the 
class numbers $\{1,1,2,2\}$ of $\O_K$. It says:

\noindent {\tt Fix $m \in \{2, 3, 5, 6\}$. Then $\P_m$ is generated by all  primitive triples $(a,b,p)$ 
where $a>0$, and $p$ is prime such that $\exists u,v\in\inte$ with $p = u^2 +mv^2$, or $2p = u^2 +mv^2$.}

Our next theorem generalizes results of \cite{Kulzer} and \cite{Baldisserri} (cf. also with \cite{Eckert}), and 
gives a free basis of the group $\P_m$ when every element of the ideal class group $Cl(K)$ has order 
at most 2. Approach of Baldisserri in \cite{Baldisserri} covered several particular cases of $m$ (see Theorem 2 on 
page 305 of \cite{Baldisserri}). In all those 36 cases $Cl(K)$ is either trivial or isomorphic to the direct product 
of several copies (up to 4) of $\inte_2$.

\begin{thm}
If $L_0=L$, then the group $\P_m,~\mbox{with}~m>3,$ is generated freely by the set of all primitive triples 
$(a,b,p)$ or $(a,b,2p)$, where $p\in L$ and 
$a$ and $b$ are positive integers such that $a^2 + mb^2 = p^2$ or $a^2 + mb^2 = (2p)^2$.
\end{thm}
\begin{proof}
First let's prove that if $p\in L_0$, then there exists a unique coprime pair $(a,b)\in\natu\times\natu$ s.t. either $a^2 + mb^2 = p^2$ 
or $a^2 + mb^2 = (2p)^2$, but not both. If $p\in L_0$, then from our lemma 1 we deduce that 
there exist integers $u$ and $v$ s.t. either $u^2+mv^2=p^2$ or $u^2 + mv^2 = (2p)^2$. If we assume that $p$ is odd and 
there are two pairs of relatively prime integers $(a,b)$ and $(u,v)$ s.t. $a^2+mb^2=p^2$ and at the same time $u^2+mv^2=(2p)^2$, 
we consider two new elements of $\P_m$:
$$
[u,v,2p]\pm [a,b,p] = \left\{
\begin{array}{lc}
[au - mbv, ~av + bu, ~2p^2] \\

[-au - mbv, ~-av + bu, ~2p^2].
\end{array}\right.
$$
Since $(bu+av)(bu-av) = b^2u^2+b^2mv^2 - (a^2v^2+mb^2v^2) = p^2(4b^2-v^2)$, if $v\neq \pm 2b$ we have 
$p~|~bu+av$ or $p~|~bu-av$ or both. In the last case we'd have $p~|~2bu$, which is impossible for an odd prime and two pairs of 
relatively prime integers $(a,b)$ and $(u,v)$. If $p$ divides only $bu+av$, then $p^2~|~bu+av$, and hence $p^2~|~au-mbv$, i.e. 
$[u,v,2p]+[a,b,p] = [s,t,2]$ for some integers $s$ and $t$. Since $m>3$ we deduce that $t=0$ and $av=-bu$, 
which contradicts the assumptions. If $p$ divides only $bu-av$, a similar argument leads to the contradiction as well. 
If $v=\pm 2b$, we again have a contradiction with $(u,v)=1$. Analogously one can consider the remaining cases when $p$ is odd 
and prove the uniqueness of a coprime pair $(a,b)\in\natu\times\natu$ s.t. either $a^2 + mb^2 = p^2$ 
or $a^2 + mb^2 = (2p)^2$. If $p=2$, we could have $a^2+mb^2=16$ only for 
$[3,1,4]\in\P_7$ or for $[1,1,4]\in\P_{15}$.

Secondly, we prove that the set of all such primitive triples $(a,b,2^{\tau}p)$, with $p\in L_0$ and $\tau\in\{0,1\}$ 
generate the group $\P_m$. Take any primitive $[a,b,c]\in\P_m$ and assume that the prime decomposition of $c$ is 
$c=2^{i_0}\cdot p_1^{i_1}\cdot\ldots\cdot p_k^{i_k}$, where $i_0\geq 0$ and $p_i$ are distinct odd rational primes. 
Then by Note 3 and our lemma 1, for each $i\in\{1,\ldots,k\}$ there exist $u_i,~v_i\in\natu\times\natu$ s.t. 
$u_i^2+mv_i^2 = p_i^2$ or $u_i^2+mv_i^2 = (2p_i)^2$, that is, either $[u_i,v_i,p_i]$ or $[u_i,v_i,2p_i]$ is a generator.
Assume without loss of generality that $i_1> 0$ and consider 
\begin{equation}\label{existence}
[a,b,c]\pm [u_1,v_1,2^{\tau}\cdot p_1] = \left\{
\begin{array}{lc}
[au_1 - mbv_1, ~av_1 + bu_1, ~c\cdot 2^{\tau}\cdot p_1] \\

[-au_1 - mbv_1, ~-av_1 + bu_1, ~c\cdot 2^{\tau}\cdot p_1],
\end{array}\right.
\end{equation}
with $\tau\in\{0,1\}$. As above, we have $(bu_1+av_1)(bu_1-av_1) \equiv 0\pmod{p^2_1}$, and hence $p^2_1$ divides at least one of 
$(bu_1+av_1)$ or $(bu_1-av_1)$, or $p_1$ divides each of these numbers. Assuming the last possibility we'll get a contradiction with 
the fact that our triples on the left hand side of (\ref{existence}) are primitive. Assuming the first possibility we obtain 
that $p^2_1~|~ au_1 - mbv_1$. Then we have  
$$
[a,b,c] = [-u_1,v_1,2^{\tau}\cdot p_1] + [D,E,2^{\tau}\cdot\left(\frac{c}{p_1}\right)],
$$
which implies that 
$$
[a,b,c] = \sum\limits_{j=1}^k \pm r_j\cdot[u_j,v_j,2^{\tau_j}p_j] + [\gamma,\omega,2^n],
$$
where $\tau_j\in\{0,1\}$ and $r_j\in\{0,\ldots, i_j\}$ for each $j\in\{1,\ldots,k\}$, and $n\geq 0$. According to our 
note 3, $[\gamma,\omega,2^n]=[1,0,1]$, unless $-m\equiv 1\pmod{8}$. In the latter case, claim 1 guaranties that 
$m\in\{7,15\}$, and one can easily prove using induction and applying the divisibility argument we just used above 
to the elements 
$$
[\gamma,\omega,2^n]\pm [q,r,4] = \left\{
\begin{array}{lc}
[q\gamma - mr\omega, ~q\omega + r\gamma, ~ 2^{n+2}] \\

[-q\gamma - mr\omega, ~-q\omega + r\gamma, ~ 2^{n+2}],
\end{array}\right.
$$
that  $[\gamma,\omega,2^n]=\pm(n-1)\cdot[q,r,4]$, where $[q,r,4]=[3,1,4]$ if 
$m=7$, and $[q,r,4]=[1,1,4]$ if $m=15$ respectively.

Thirdly, we show that the generating set is free of any nontrivial relations. Suppose that there exists a primitive 
triple $[K,L,M]\in\P_m$ with two different presentations by elements from the generating set, i.e. 
\begin{equation}\label{cancel2}
\sum\limits_{j\in J} z_j= [K,L,M] = \sum\limits_{t\in T} z_t,
\end{equation}
where each $z_i=[a_i,~b_i,~2^{\tau}p_i],~\mbox{with}~\tau\in\{0,1\}~\mbox{and}~i\in\{J\cup T\}$ is a primitive triple with 
$p_i\in L_0$. We can assume without loss of generality that $J\cap T=\emptyset$. 
Since the third components get multiplied when we add two elements of $\P_m$, we can write the third components  
of the left and right hand sides of (\ref{cancel2}) (before reducing to the corresponding primitive triples) as 
$$
2^{j_0}\cdot\prod\limits_{j\in J}p_j~~~~~\mbox{and} ~~~~~2^{t_0}\cdot\prod\limits_{t\in T}p_t~~~~~\mbox{respectively.}
$$
Since the left and right hand sides of (\ref{cancel2}) are equal, and since $J\cap T=\emptyset$, we see that $M=2^n,~n\geq 0$.
Then we deduce from one of our notes above that $[K,L,M]=[1,0,1]$, unless $-m\equiv 1\pmod{8}$. 
Suppose now that $-m\equiv 1\pmod{8}$, and the third component of one of $z_j$ is divisible by an odd prime $p$. 
Then we could write the left hand 
side of (\ref{cancel2}) as $k\cdot[a,b,2^{\tau}\cdot p] + [U,V,W]$, where $k\geq 1$ and $\tau\in\{0,1\}$. Writing 
$k\cdot[a,b,2^{\tau}\cdot p] $ as $[\alpha,\beta,2^{\varepsilon}\cdot p^k]$ (see Lemma 2. below), we can rewrite 
the left hand side of (\ref{cancel2}) as $[\alpha,\beta,2^{\varepsilon}\cdot p^k] + [U,V,W]$ where the triples 
$(\alpha,\beta,2^{\varepsilon}\cdot p^k)$ and $(U,V,W)$ are not necessarily primitive, but $(p,~W)=1$  and 
$(\beta, p)=1$. Thus 
$$
[\alpha,\beta,2^{\varepsilon}\cdot p^k] + [U,V,W] = 
[\alpha U-m\beta V,~\alpha V + \beta U,~2^{\varepsilon}\cdot p^k\cdot W] = [K,~L,~2^n],
$$
which implies that $p~|~\alpha U-m\beta V$, and also $p~|~\alpha V + \beta U$. In such situation we would also 
have $p~|~W$, which is not allowed. Thus the left and right hand sides of (\ref{cancel2}) 
could only have group elements where the third component is a power of 2. But there is only one such basis element, 
which is $[*,~*,~4]$, and hence we can not have two different nontrivial presentations of $[K,L,M]$. 

\end{proof}

\begin{lem}
Suppose for an odd prime $p$ we have a primitive triple $[u,v,p^bw]\in\P_m$ s.t. $(p,w)=1$, and $b\geq 1$. 
Take any $n\in\natu$. If $[u_n,v_n,w_n]=n\cdot[u,v,p^bw]$ with $(u_n,v_n)=1$, then $p^{nb}~|~w_n$.
\end{lem}
\begin{proof}
First observe that $[u_2,v_2,w_2]= [u^2-mv^2,2uv,p^{2b}w^2]$ and if $p^{2b}$ doesn't divide $w_2$ 
we must have $p~|~2uv$, which is not allowed. Let us write $u^2-mv^2=(p^bw)^2-2mv^2=p\cdot A_2 -2mv^2$ for some 
$A_2\in\inte$. Then one can use induction to show that 
$$
[u_{2k},~v_{2k},~w_{2k}] = [p\cdot A_{2k} \pm 2^{2k-1}\cdot m^k\cdot v^{2k},~p\cdot B_{2k} - \pm 
2^{2k-1}\cdot m^{k-1}\cdot u\cdot v^{2k-1},~ p^{2kb}\cdot w^{2k}]
$$
with both $A_{2k}~\mbox{and}~B_{2k}\in\inte$. Analogously one proves a similar formula for $n=odd$, 
where the second component of $(2k+1)\cdot[u,v,p^bw]$ will be $v_{2k+1} = p\cdot B_{2k+1} \pm 2^{2k}\cdot m^k\cdot v^{2k+1}$. 
Clearly, in both cases these formulas imply that $p^{nb }~|~w_n$.
\end{proof}

\begin{example} Let $m=35$, then $Cl(K)\cong C_2$ and we have 
$$
L_0=L=\{3, 11, 13, 17, 29, 47, 71, 73, 79, 83,  97, 103, 109, 149, 151, \ldots\}.
$$
Here are first few generating triples which satisfy equation $a^2+mb^2=p^2$:
$$
\{[1, 12, 71],~[17, 12, 73], ~[43, 12, 83],~ [131, 12, 149],~\ldots\},
$$
and here are several first generating triples satisfying equation  $a^2+mb^2=(2p)^2$:
$$
\{[1, 1, 2\cdot3], [13, 3, 2\cdot11], [19, 3, 2\cdot13],[29, 3, 2\cdot17], [23, 9, 2\cdot29], [31, 15, 2\cdot47],[157, 3, 2\cdot79], \dots\}.
$$
\end{example}

\section{Ideal classes of higher order}

From now on I will assume that $Cl(K)$ has elements of order higher than 2 and hence $E$ will be a proper subgroup 
of $Cl(K)$. In such a case we can extend the diagram defining map $f$ to the following one.
\begin{equation}\label{diagram2}
\xymatrix{
& P(\O_K) \ar[r]_-{\mu}  \ar[d]_{\pi} & \rP\ar@/^/[dl]|-{f} \ar@/_1pc/@{-->}[l]_-l \ar[d]^{g}& \\
E ~ \ar@{^{(}->}[r] & Cl(K) \ar[r]^{\rho} & Cl(K)/E \ar@{->>}[r] & 0
}
\end{equation}
Here map $\rho$ is the canonical projection onto the factor group $Cl(K)/E$, and the map $g:=\rho\circ f$. 

Since $Cl(K)$ is finite, we can present the factor group $Cl(K)/E$ as a direct sum of cyclic groups
$Cl(K)/E\cong G_1\times G_2 \times\cdots\times G_n$ with corresponding orders denoted by $h_j :=|G_j |$, 
$j\in\{1,2,\ldots,n\}$. For each $G_j$ pick an ideal $P_j\in P(\O_K)$, s.t. 
$G_j = \langle \rho\circ\pi(P_j)\rangle$, (c.f. with the presentation of the 
factor group $Cl(S)/M$ as a direct sum of cyclic groups and the corresponding construction right after that on page 
84 of \cite{Zanardo}). 

\noindent Now let's use the lifting $l:\rP \to P(\O_K)$ to construct a map $\beta:L\to \P_m$ as follows. 
First, take any prime $p\in L\setminus L_0$.

{\bf Case 1).} Suppose that $\mu(P_j) = p$ for some $j\in\{1,2,\ldots, n\}$. Then $(\pi(P_j))^{h_j}\in E$ and by 
Lemma 1 above we can write $(2^{1-\delta}p^{h_j})^2 = u^2 + mv^2$ for some positive integers $u$ and $v$ 
(where as above, $\delta \in \{0,1\}$). If we suppose that $p~|~\gcd(u,v)$ we will deduce that $P'_j$ (the conjugate ideal of $P_j$) 
divides in $\O_K$ a power of $P_j$, which is impossible, therefore we can assume that $\gcd(u,v) = 1$. 
Thus we obtained a primitive almost pythagorean
triple $(u,v,2^{1-\delta}\cdot p^{h_j})$, which is unique up to signs of $u$ and $v$. Therefore we define 
$$
\beta(p):=[u,~v,~2^{1-\delta}\cdot p^{h_j}],~~~\mbox{where}~~~(u,v)\in\natu\times\natu.
$$

{\bf Case 2).} Suppose further that $p\in L\setminus(L_0\cup \left\{\mu(P_j)~|~j\in\{1,2,\ldots,n\}\right\})$. Then for 
the prime ideal $P=l(p)$ there exist nonegative integers $a_{jp}\in\{0,1,\ldots,h_j-1\}$ such that
\begin{equation}\label{case2}
P\cdot P_1^{a_{1p}}\cdot\ldots\cdot P_n^{a_{np}} \in E,~~\mbox{and hence}~~
P^2\cdot P_1^{2a_{1p}}\cdot\ldots\cdot P_n^{2a_{np}} = \left\langle \frac{u+v\sqrt{-m}}{2^{1-\delta}}\right\rangle
\end{equation}
for some $(u,v)\in\inte\times\inte$. 
\begin{claim}
$$
\gcd\left(\frac{u}{2^{1-\delta}},\frac{v}{2^{1-\delta}}\right) = 1.
$$
\end{claim}
\begin{proof}
Suppose that we have a prime $q$, which splits and divides both $u$ and $v$. Then if $Q:=l(q)$, we have 
$\langle q\rangle = Q\cdot Q' ~|~P^2\cdot P_1^{2a_{1p}}\cdot\ldots\cdot P_n^{2a_{np}}$. But this is 
impossible since none of the idelas $\{P,~P_1,~\ldots,~P_n\}$ has order 2 in $Cl(K)$ and they all have different norms, 
while $N(Q)=N(Q')=q$. 
\end{proof}
Thus, as in the first case, we found a relatively prime pair of integers $(u,v)$ s.t. 
$$
u^2 + mv^2 = (2^{1-\delta}\cdot p\cdot p_1^{a_{1p}}\cdot\ldots p_n^{a_{np}})^2.
$$ 
Note that this time, the pair $(u,v)$ is not necessarily unique. If, for example, $h_1=2t$, and we found a triple 
$(u_1,~v_1,~2^{1-\delta}\cdot p\cdot p_1^t)$, then for the triple $(u,~v,~2^{1-\delta}\cdot p^{h_1})$ found in case 1) above,
one of the triples 
$$
[u_1,~v_1,~2^{1-\delta}\cdot p\cdot p_1^t] \pm [u,~v,~2^{1-\delta}\cdot p^{h_1}]
$$
would be another primitive almost pythagorean triple with the third component $2^{1-\delta}\cdot p\cdot p_1^t$ and the first  
two components different from $\pm u_1$ and $\pm v_1$. We can alternate the construction giving (\ref{case2}) and show that 
each $a_{jp}$ may be taken from the set $\{0,1,\ldots, \left[\frac{h_j}{2}\right]\}$. Indeed if, for example 
$a_{1p}>\left[\frac{h_1}{2}\right]$, we can multiply the product
$P\cdot P_1^{a_{1p}}\cdot\ldots\cdot P_n^{a_{np}} $ by $\langle p_1\rangle^{h_1-a_{1p}}=(P_1\cdot P_1')^{h_1-a_{1p}}$ and 
obtain an element 
$$
P\cdot (P'_1)^{b_{1p}}\cdot\ldots\cdot P_n^{a_{np}} \in E~~~\mbox{with}~~~b_{1p} = h_1-a_{1p} \leq\left[\frac{h_1}{2}\right].
$$

Now we are ready to define map $\beta: L\setminus L_0 \to \P_m$ for 
$p\in L\setminus(L_0\cup \left\{\mu(P_j)~|~j\in\{1,2,\ldots,n\}\right\})$. Take any such prime number $p$, then as we just 
explained above, there exist relatively prime integers $u$ and $v$ such that 
$u^2 + mv^2 = (2^{1-\delta}\cdot p\cdot p_1^{a_{1p}}\cdot\ldots p_n^{a_{np}})^2$, and $\{a_{1p},\ldots,a_{np}\}$ are the 
coordinates of the ideal $P=l(p)$ in $Cl(K)/E$ written in the basis $\{\rho\circ\pi(P_1)~or~\rho\circ\pi(P'_1),~\ldots,~\rho\circ\pi(P_n)~
or~\rho\circ\pi(P'_n)\}$, which means that 
$a_{jp}\leq\left[\frac{h_j}{2}\right]$ for each $j$. Among such primitive triples
$(u,~v,~2^{1-\delta}\cdot p\cdot p_1^{a_{1p}}\cdot\ldots p_n^{a_{np}})$ take the one with the smallest value of $|u|$, 
say $u_0$, and define 
$$
\beta(p):=[u_0,~v_0,~2^{1-\delta}\cdot p\cdot p_1^{a_{1p}}\cdot\ldots p_n^{a_{np}}],~~~\mbox{where}~~~(u_0,~v_0)\in\natu\times\natu.
$$
We can extend map $\beta$ to elements of $L_0$ as well, by defining $\beta(p):=[a,b,p]$ if $a^2+mb^2=p^2$, and 
$\beta(p):=[a,b,2p]$ if $a^2+mb^2=(2p)^2$, when $p\in L_0$. Thus we obtain a map  $\beta:~L\lra \P_m$.

\begin{note}
It is clear from the construction that this map $\beta:~L\lra \P_m$ is one-to-one. 
It is also clear that $\beta$ depends on several choices we've made in the construction, 
and in particular it depends on the choice of the prime 
ideals projecting onto the generators of subgroups $G_j$ of $Cl(K)/E$.
\end{note}

\begin{example} Let $m=23$, then $Cl(K)\cong C_3$ and we have 
$$
L=\{2, 3, 13, 29, 31, 41, 47, 59, 71, 73, 101, 127, 131, 139, 151, 163, 167, 173, 179, 193, 197, \ldots\}.
$$
Here $L_0 = \{59, 101, 167, 173,\ldots \}\subset L$ with the corresponding values
$$
\beta(59) =[13, 12, 59], ~\beta(101)=[83, 12, 101],~ \beta(167)=[121, 24, 167],~\beta(173)= [11, 36, 173], \ldots
$$ 
If we choose the ideal $\left\langle 2, \frac{1+\sqrt{-23}}{2}\right \rangle$ lying over 2 as the one, which gives 
the generator $P_1$ of $Cl(K)/E=C_3$, then $\beta(2) = [7,3,2\cdot 2^3]$, and also
$$
\beta(3)=[11, 1, 2\cdot(3\cdot2)],~\beta(13)=[29, 9, 2\cdot(13\cdot2)],~\beta(29)=[91, 15, 2\cdot(29\cdot2)], ~
\ldots~\mbox{and so on}~\ldots
$$
If we choose the ideal $\left\langle 3, 1+\sqrt{-23} \right \rangle$ lying over 3 as the one, which gives 
the generator of $Cl(K)/E=C_3$, then $\beta(3) = [19,4, 3^3]$, and also
$$
\beta(2)=[11, 1, 2\cdot(2\cdot3)],~\beta(13)=[7, 8, (13\cdot3)],~\beta(29)=[41, 16, (29\cdot3)], ~
\ldots~\mbox{and so on}~\ldots
$$
\end{example}

Using this map $\beta$ we can now describe a basis of the 
group $\P_m$ for all $m>3$ when the corresponding ideal 
class group $Cl(K)$ has elements or order higher than 2 
(cf. with Theorem 3 of \cite{Zanardo}).

\begin{thm}
$Im(\beta)\subset \P_m$ forms a basis of the free abelian group $\P_m,~m>3$.
\end{thm}
\begin{proof}
First I explain why $Im(\beta)$ is a generating set. Take any primitive $[a,b,c]\in\P_m$ and write the prime decomposition 
in $\O_K$ for the ideal 
$$
\langle a - b\sqrt{-m}\rangle = T_1^{l_1}\cdot T_2^{l_2}\cdot\ldots\cdot T_r^{l_r} = : \T ,~~\mbox{and note that}~2\mid l_j~\forall j,~ \mbox{since}~(a,b)=1.
$$
If $T_1$ is not one of the chosen generators $\{P_1,P_2,\ldots,P_n\}$ for $Cl(K)/E$, then using the construction we used in case 2), 
we can write that 
$T_1\doteq Q_1^{b_{1t_1}}\cdot Q_2^{b_{2t_1}}\cdot\ldots \cdot Q_n^{b_{nt_1}}$, where  $\doteq$ means the equality up to a 
product by an element of $E$, $t_1=N(T_1)$ and each $Q_j\in\{P_j,~P'_j\}$ 
so that $b_{jt_1}\in\{0,1,\ldots,\left[\frac{h_j}{2}\right]\}$. Denote the product 
$T_1'\cdot Q_1^{b_{1t_1}}\cdot Q_2^{b_{2t_1}}\cdot\ldots \cdot Q_n^{b_{nt_1}}$ by $\T_1$ and we obtain 
$$
\T\cdot \T_1^{l_1} \doteq T_2^{l_2}\cdot\ldots\cdot T_r^{l_r}\cdot (Q_1^{b_{1t_1}}\cdot Q_2^{b_{2t_1}}\cdot\ldots \cdot Q_n^{b_{nt_1}})^{l_1}
$$
Similarly we can ``eliminate'' from $\T$ all $T_j$, which are not in the set $\{P_1,P_2,\ldots,P_n\}$. Suppose that on the 
other hand we have $T_1=P_1$ (we can always rename $T_j$ if needed). Then using division with remainder we can write 
$l_1 = q_1\cdot h_1 + r_1$, with $0\leq r_1< h_1$ and hence $(T'_1)^{l_1} = \T^{q_1}_1\cdot Q^{r_1}_1$ where $Q_1 = P'_1$, 
and $\T_1= (P'_1)^{h_1}$, which also would give us 
$$
\T\cdot \T_1^{q_1} \doteq T_2^{l_2}\cdot\ldots\cdot T_r^{l_r}\cdot (Q'_1)^{r_1}.
$$
Therefore if we let $\omega_j=q_j$ when $T_j=P_j$ and $l_j=q_j\cdot h_j + r_j$, and $\omega_j=l_j$ otherwise, we deduce that 
$$
\T\cdot\T_1^{\omega_1}\cdot\ldots\cdot \T_r^{\omega_r} \in E,
$$
since each $\T_j\in E$ and $\T$ is principal. Hence we can write for the conjugate ideal 
$$
\langle a + b\sqrt{-m}\rangle 
= \T_1^{\omega_1}\cdot\ldots\cdot \T_r^{\omega_r} \cdot I,~~\mbox{for some}~~I\in E.
$$
Since each $l_j$ is even, this last equality implies 
$$
[a,b,c] = \gamma_1\beta(t_1) + \gamma_2\beta(t_2) + \ldots + \gamma_r\beta(t_r) + [g,h,d], ~~\mbox{with all}~~\gamma_j\in\inte,
$$
and where $d$ can be divided only by 2, primes from the set $L_0$, and powers of $(p_j)^{h_j}$ for $j\in\{1,2,\ldots,n\}$ (recall that 
by our choice, $Cl(K)/E$ is generated by $\rho\circ\pi(P_j)$ where $\mu(P_j)=p_j$ and $h_j$ is the order of the corresponding 
cyclic subgroup $G_j$). Repeating the second step from the proof of 
Theorem 1 above, and presenting $[g,h,d]$ as a linear combination of elements from $\beta(L_0)$, and from 
$\{\beta(p_1),\ldots \beta(p_n)\}$ we conclude that $\P_m$ is generated by the image $Im(\beta)$.
 
To show that the set $Im(\beta)$ is free of any nontrivial relations I will use the same approach that was used in the third 
part of the proof of Theorem 1. Let's suppose that there is a nontrivial relation among the generators, i.e. there exists a 
primitive triple $(K,L,M)$ s.t. $ [K,L,M]\in\P_m$ and 
\begin{equation}\label{last}
\sum\limits_{i\in I} s_i\cdot\beta(p_i) = [K,L,M] = \sum\limits_{j\in J} t_j\cdot\beta(p_j),~~
\mbox{and}~~I\cap J = \emptyset,
\end{equation}
where $\{p_i~|~i\in I\}\cup \{p_j~|~j\in J\} \subset L$, and $s_i,~t_j \in\inte$ for all $i\in I$ and $j\in J$. 
If we assume that there exists an odd prime $q\in \{p_i~|~i\in I\}$, 
which does not lie below any of the chosen idelas $P_1,\ldots, P_n$ mapping into the 
generators of $Cl(K)/E$, then using our lemma 2, we can write the left hand side of (\ref{last}) as
$$
\sum\limits_{i\in I,~p_i\neq q} s_i\cdot\beta(p_i) + s\cdot \beta(q) =  [K_1,L_1,M_1] + [u,~v,~q^s\cdot w],
$$
where the triples $(K_1,L_1,M_1)$ and $(u,v,q^s\cdot w)$ are primitive. It follows from the definition of operation 
in $\P_m$, fact that $I\cap J=\emptyset$, and our construction of map $\beta$ that $(M_1,q)=1$ and also 
$(M,q)=1$. But 
$$
[K,L,M] = [K_1,L_1,M_1]+ [u,v,q^s\cdot w] = [uK_1 - mvL_1,~vK_1+ uL_1, ~q^s\cdot w\cdot M_1],
$$
which is possible only when $q^s$ divides both $uK_1 - mvL_1$, and $vK_1+ uL_1$. In such case $q^s$ will also divide 
$vK_1^2 + mvL_1^2= v\cdot M_1^2$, which is possible only if $s=0$. This proves that if a nontrivial relation exists 
among elements of $Im(\beta)$, it could only involve elements from $\beta(L_0)$ and from 
$\{\beta(2),\beta(p_1),\ldots, \beta(p_n)\}$.

If we assume that there is a nontrivial relation (\ref{last}), with $\{p_i~|~i\in I\}\cup \{p_j~|~j\in J\} \subset L_0\cup \{p_1,\ldots,p_n\}$, 
and, for example, that odd $p_1\in \{p_i~|~i\in I\}$, then using the definition of map $\beta$ and lemma 2, 
we can rewrite the left hand side of (\ref{last}) this time as 
$$
\sum\limits_{i\in I,~p_i\neq p_1} s_i\cdot\beta(p_i) + s_1\cdot \beta(p_1) =  [K_1,L_1,M_1] + [u,~v,~p_1^{s_1h_1}\cdot w],
$$
where $(K_1,L_1,M_1)$ and $(u,v,p_1^{s_1h_1}\cdot w)$ are primitive triples such that $(M_1,p_1)=1$. Since 
$p_1\notin  \{p_j~|~j\in J\}$, we also have $(M,p_1)=1$. We will come to a contradiction again using  
exactly the same argument we just used above by presenting $[K,L,M]$ as $[*,~*,~ p_1^{s_1h_1}\cdot w\cdot M_1]$.
Hence the elements from $\{\beta(p_1),\ldots, \beta(p_n)\}$ can not be involved in a nontrivial relation either
(with only possible exception of $\beta(2)$). 

Finally, assuming that there is a nontrivial relation among elements of $\{\beta(p)~|~p\in L_0\cup\{2\}\}$ only, one can 
repeat the third part of the proof of Theorem 1 to deduce that the left and right hand sides of (\ref{last}) 
could only have group elements where the third component is a power of 2. But there could be only one such basis element, 
which is $\beta(2)$, and hence we can not have two different nontrivial presentations of $[K,L,M]$. This finishes the proof of 
Theorem 2.
\end{proof}

\noindent Here is an example, which motivates and illustrates the approach I used above in the construction of map $\beta$.

\begin{example} Let $m=974$, then $Cl(K)\cong C_{12}\times C_3$. Chose the ideal $\langle 5, 1+\sqrt{-m}\rangle$ 
as the generator $P_1$ of $C_{12}$ and the ideal $\langle 41, 16 +\sqrt{-m}\rangle$ 
as the generator $P_2$ of $C_3$. Then $\rho\circ\pi(P_1)$ and $\rho\circ\pi(P_2)$ will be the chosen generators 
of the factor group $CL(K)/E\cong C_6\times C_3$. Using a computer, one can easily find that subset $L_0$ starts 
with numbers $L_0=\{937,~983,~\ldots\}$, and correspondingly  
$\beta(937)=[37, 30, 937]$ and $\beta(983) = [965, 6, 983]$. On the other hand, the set $L$ contains many smaller primes as well: 
$$
L=\{3, 5, 11, 13, 31, 37, 41, 43, 59, 71, 73, 89, 97, 101, 103, 109, 127, 131, 137, 149, 163, \ldots\}.
$$
For the ideal $P_2$ we have $P^6_2 = \langle 61129 - 1020\sqrt{-974}\rangle$ and $\beta(41) = [61129, 1020,41^3]$ 
with $61129^2+974\cdot 1020^2 = (41^3)^2$. For the ideal $P_1$ we have 
$P^{12}_1 = \langle 14651 -174\sqrt{-974} \rangle$ and $\beta(5) = [14651, 174,5^6]$. 

Now let's chose a prime which is not in $L_0\cup\{\mu(P_1),\mu(P_2)\}$, for example $p=3$. 
Then we have the product $P\cdot P_1\cdot P_2\in E$ since 
$$
\langle 3, 1+\sqrt{-974}\rangle\cdot\langle 5, 1+\sqrt{-m}\rangle \cdot \langle 41, 16 +\sqrt{-m}\rangle = 
\langle 615, 16 + \sqrt{-974} \rangle, 
$$
and $\langle 615, 16 + \sqrt{-974} \rangle^2 = \langle 359 -16\sqrt{-974} \rangle$. Hence 
$\beta(3)=[359, 16, 3\cdot 5\cdot 41]$. 

\noindent Notice also that we have two primitive triples for $p=37$ 
where the third components are the same:
$$
[4141,66,37\cdot 5^3]~~\mbox{with}~~(\langle 37, 5+\sqrt{-974}\rangle\cdot P_1^3)^2 = 
\langle 4141 + 66\sqrt{-974}\rangle,
$$
and 
$$
[3167, 108, 37\cdot 5^3]~~\mbox{with}~~(\langle 37, 32+\sqrt{-974}\rangle\cdot P_1^3)^2 
= \langle 3167 - 108\sqrt{-974}\rangle.
$$
Since 
$$
\langle 4141 + 66\sqrt{-974}\rangle\cdot  \langle 14651 +174\sqrt{-974} \rangle =  
\langle 3167 + 108\sqrt{-974}\rangle\cdot \langle 5^6\rangle
$$
we have 
$$
[4141,~66,~37\cdot 5^3] + [14651,~174,~5^6] = [3167,~108, ~37\cdot 5^3],
$$
and hence according to our definition we have $\beta(37) = [3167,~108, ~37\cdot 5^3]$.
\end{example}

\end{document}